\documentclass[11pt]{amsart}

\usepackage{amssymb,amsthm,amsmath,amscd}
\usepackage[backref=page, colorlinks=true, linkcolor=magenta, citecolor=cyan]{hyperref}
\usepackage{marvosym}

\usepackage{mathtools}
\usepackage{soul}

\RequirePackage[dvipsnames,usenames]{color}

\input{kmacros3.sty}

\usepackage{tikz}
\usepackage{tikz-cd}
\usetikzlibrary{cd}
\usepackage{graphicx}
\usepackage[all,cmtip]{xy}
\usepackage{mathpazo}

\usepackage{mabliautoref}
\numberwithin{equation}{theorem}

\usepackage{bm}
\usepackage{pifont}
\usepackage{upgreek}

\usepackage{eucal}
\usepackage{ulem}
\usepackage{stmaryrd}

\numberwithin{equation}{theorem}

\usepackage{fullpage}

\usepackage{setspace}
\usepackage[shortlabels]{enumitem}
\usepackage{calc}
\usepackage{colonequals}

\usepackage{verbatim}
\usepackage{alltt}

\newcommand{\ins}{\subseteq}
\newcommand{\inj}{\xhookrightarrow{}}

\begin{document}

\title{\emph{F}-intersection flatness of dagger and Berkovich Tate algebras}

% \author{Author 1}
% \author{Author 2}
% \author{Author 3}

\author{Rankeya Datta}
\address{Department of Mathematics, University of Missouri, 
Columbia, Missouri 65212, USA}
\email{rankeya.datta@missouri.edu}
\thanks{The first author was supported in part by NSF DMS Grant \#2502333 and Simons Foundation grant MP-TSM-00002400.}

\author{Jack J Garzella}
\address{Department of Mathematics, University of California San Diego, La Jolla, California, 92092, USA}
\email{jgarzell@ucsd.edu}
\thanks{The second author was supported in part by the NSF GRFP under Grant \#2038238, and a fellowship from the Sloan Foundation}

\author[Tucker]{Kevin Tucker}
\address{Department of Mathematics, University of Illinois at Chicago, Chicago, IL, USA}
\email{kftucker@uic.edu}
\thanks{The third author was supported in part by NSF DMS Grants \#2200716 and \#2501904.}

\begin{abstract}
    We show, using the techniques developed in \cite{DESTate,DattaEpsteinTuckerML}, that dagger Tate algebras and Tate algebras in the sense of Berkovich in prime characteristic \(p > 0\) have intersection flat Frobenius. Equivalently, if $S$ is such a ring, then $S^{1/p}$ is a flat and Mittag-Leffler $S$-module. As a consequence, we deduce that any ideal-adic completion of a reduced ring that is essentially of finite type over a dagger Tate algebra or a Berkovich Tate algebra in prime characteristic has big test elements.\vspace{-2ex}
    %Along the way, we show that dagger algebras are excellent in characteristic \(p\).
\end{abstract}

\maketitle

\section{Introduction}

% A long-standing conjecture in prime characteristic commutative algebra is the existence of \emph{big test elements} for an excellent reduced ring $R$ of prime characteristic $p > 0$ \cite{HunekeTightClosureBook}. Big test elements are uniform multipliers in $R$ that can be used to `test' when an element of an arbitrary $R$-module $M$ is in the tight closure $N_M^*$ of any submodule $N$. These special multipliers exist for many commonly occurring classes of excellent rings such as the class of reduced Noetherian rings with finite Frobenius (\emph{$F$-finite} rings) and the class of reduced rings that are essentially of finite type over an excellent local ring of characteristic $p$ \cite{HochsterHunekeTC1,HochsterHunekeFRegularityTestElementsBaseChange}. In dimension $\leq 2$, big test elements are known to exist in full generality \cite{AberbachUniformArtinRees}. In addition, big test elements also exist for excellent rings with universally injective Frobenius \cite{SharpAnExcellentFPureRing}, commonly known as \emph{$F$-pure rings}.

A long-standing conjecture in prime-characteristic commutative algebra asserts that every excellent reduced ring \(R\) of characteristic \(p>0\) has \emph{big test elements} \cite{HunekeTightClosureBook}. These are uniform multipliers in \(R\) that test, for every \(R\)-module \(M\) and submodule \(N\), whether a given element lies in the tight closure \(N_M^*\). Their existence is known for reduced Noetherian rings with finite Frobenius (\emph{\(F\)-finite rings}) and for reduced rings essentially of finite type over an excellent local ring of characteristic \(p\) \cite{HochsterHunekeTC1,HochsterHunekeFRegularityTestElementsBaseChange}. It is also known for all excellent reduced rings of dimension at most \(2\) \cite{AberbachUniformArtinRees} and for excellent rings with universally injective Frobenius, commonly called \emph{\(F\)-pure rings} \cite{SharpAnExcellentFPureRing}.

The main result of this paper establishes the existence of big test elements for excellent rings arising in the theories of Berkovich spaces \cite{BerkovichSpectralTheory,BerkovichEtale} and dagger spaces \cite{GKDaggeralgebras,le-stum-2022-rigid}; see \autoref{thm:intro-main-thm} for a precise statement. To motivate the result, let \((k,|\cdot|)\) be a non-Archimedean field of characteristic \(p\). Although reduced affinoid \(k\)-algebras are excellent \cite{KiehlTate}, they need not be \(F\)-finite, essentially of finite type over an excellent local ring, or \(F\)-pure. Thus, the results above do not settle whether such rings have big test elements. Moreover, the known proofs for reduced \(F\)-finite Noetherian rings and reduced rings essentially of finite type over an excellent local ring of characteristic \(p\) rely on a nonzero \(R\)-linear map \(R^{1/p}\to R\) \cite{HochsterHunekeTC1,HochsterHunekeFRegularityTestElementsBaseChange}. More generally, Aberbach showed that any excellent reduced ring \(R\) admitting such a map has test elements \cite{AberbachUniformArtinRees}. However, the first author and Murayama constructed non-Archimedean fields \(k\) of characteristic \(p>0\) over which there exist reduced affinoid algebras with no such nonzero maps \cite{DattaMurayamaTate}. Consequently, these map-based strategies do not appear to extend to show the existence of big test elements for reduced affinoid algebras.

%which \(T_n(k)^{1/p}\) admits no nonzero \(T_n(k)\)-linear map to \(T_n(k)\) for any \(n>0\) \cite{DattaMurayamaTate}. Consequently, these map-based strategies do not appear to extend to homomorphic images of \(T_n(k)\), namely, affinoid algebras.

% A different path forward is suggested by the work of Sharp \cite{SharpBigTestElements}, who shows that if a reduced $R$ is a homomorphic image of an excellent regular ring $S$ of characteristic $p$ such that $S^{1/p}$ as an $S$-module satisfies a technical property known as \emph{intersection flatness}, then $R$ has big test elements. For the expert, we note that the intersection flatness condition is equivalent to $S^{1/p}$ being a flat and Mittag-Leffler $S$-module (see \cite{DESTate,DattaEpsteinTuckerML}).  In \cite{DESTate}, the first and third authors, Epstein and Schwede use Sharp's result to show that $T_n(k)^{1/p}$ is intersection flat over $T_n(k)$ for any non-Archimedean field $k$ of characteristic $p$, thereby establishing that reduced affinoid algebras of prime characteristic have big test elements. Their method uses functional analytic properties of non-Archimedean fields, topological properties of Tate algebras and a robust theory of intersection flatness developed by the first and third authors and Epstein \cite{DattaEpsteinTuckerML}, building on and unifying results obtained in \cite{rg71,OhmRushcontent,HochsterJeffriesintflatness}.

Sharp's work provides a different route \cite{SharpBigTestElements}. He proved that if a reduced ring \(R\) is a homomorphic image of an excellent regular ring \(S\) of characteristic \(p\) and \(S^{1/p}\) is intersection flat as an \(S\)-module, then \(R\) has big test elements. Intersection flatness is equivalent to \(S^{1/p}\) being a flat and Mittag-Leffler \(S\)-module \cite{DESTate,DattaEpsteinTuckerML}. Using Sharp's theorem, the first and third authors, together with Epstein and Schwede, proved that for every non-Archimedean field \(k\) of characteristic \(p\) and integer $n > 0$, the Tate algebra $T_n(k)$ has the property that $T_n(k)^{1/p}$ is an intersection flat $T_n(k)$-module, thereby establishing the existence of big test elements for reduced affinoid algebras  of prime characteristic \cite{DESTate} since the latter by definition are homomophic images of $T_n(k)$. Their proof combines functional-analytic properties of non-Archimedean fields and topological properties of Tate algebras with the theory of intersection flatness developed by the first and third authors and Epstein \cite{DattaEpsteinTuckerML}. The latter builds on and unifies results from \cite{rg71,OhmRushcontent,HochsterJeffriesintflatness}.

The affinoid algebras alluded to above occur in Tate's theory of \textit{rigid analytic spaces} \cite{TateRigid}.
%\footnote{Though this theory of analytic geometry is not as widely used anymore, an example of a modern theory that uses Tate algebras is Huber's theory of \textit{adic spaces}.}. 
However, there are other theories of analytic geometry that give rise to excellent rings that are not affinoid in the sense of Tate. First, there is the theory of \textit{Berkovich Spaces}, in which Berkovich defines a variant \(T_{n,\rho}(k)\) of the Tate algebra, for any $n$-tuple of positive real numbers $\rho = (\rho_1,\dots,\rho_n)$ \cite{BerkovichSpectralTheory,BerkovichEtale}. If \(\rho \in (\Gamma_{> 0})^n\), where \(\Gamma \coloneqq |\overline{k}^\times|\) is the value group of an algebraic closure \(\overline{k}\) of $k$, then \(T_{n,\rho}(k)\) is affinoid in the sense of Tate. Moreover, if $\rho = (1,\dots,1)$, then $T_{n,\rho}(k) = T_n(k)$, the classical Tate algebra. However, if \(\rho \notin (\Gamma_{> 0})^n\), then $T_{n,\rho}(k)$ is a new type of excellent \cite{DucrosBerkovichExcellent} regular ring. We will call homomorphic images of \(T_{n,\rho}(k)\) \emph{affinoid algebras in the sense of Berkovich}, or simply \textit{Berkovich affinoid algebras}. Berkovich spaces have found applications throughout rigid geometry, including in the proof of the local Langlands theorem \cite{harris-taylor-2001-local-langlands}, the resolution of singularities \cite{TemkinInseparableUniformization}, and early work on perfectoid geometry \cite{kedlaya-liu-2015-relative-p-adic}.

Secondly, related to the theory of Berkovich spaces is the theory of \textit{dagger spaces}, or analytic spaces with overconvergent structure sheaves. The algebraic building blocks of dagger spaces are the rings \(T_{n}(k)^{\dagger} \coloneqq \colim_{r \in \Gamma_{>1}} T_{n,\rho_r}(k)\), where $\rho_r = (r,\dots,r)$. These rings appear to have been first considered by G{\"u}ntzer \cite{GuntzerOverconvergent} and can be interpreted as rings of analytic functions which converge on a polydisk of some radius larger than $1$. Homomorphic images of \(T_{n}(k)^{\dagger}\) are called \textit{dagger affinoid algebras}. Unlike \(T_{n}(k)\) and \(T_{n,\rho}(k)\),  $T_n(k)^\dagger$ is not a Banach $k$-algebra since it is not complete with respect to the canonical Gauss norm that one can define on it.  Nevertheless, one can use dagger affinoid algebras to mimic other constructions in non-Archimedean analytic geometry. For example, Grosse-Kl\"onne describes
a version of the theory of Tate's rigid analytic
spaces in terms of dagger affinoid algebras \cite{GKDaggeralgebras}
%\footnote{Another version of the theory of adic spaces with overconvergent structure sheaf is described in forthcoming work of Shimizu.}
. Dagger affinoid algebras have been used in the theory of rigid cohomology \cite{le-stum-2007-rigid, le-stum-2022-rigid}, and for computing zeta functions of varieties of finite fields \cite{kedlaya-2001-algorithm}.

Our main result shows that big test elements exist for reduced Berkovich affinoids and reduced dagger affinoids. 

\begin{theorem}(\autoref{thm:Dagger-Tate-FIF}, \autoref{cor:Test-elements-dagger-affinoid}, \autoref{thm:berktate:intflat}, \autoref{cor:berktate:FIF})
\label{thm:intro-main-thm}
Let $(k,|\cdot|)$ be a non-Archimedean field of characteristic $p > 0$. Then for all integers $n > 0$, every $\rho \in (\mathbb{R}_{>0})^n$, and indeterminates $\underline{x} = x_1,\dots,x_r$ and $\underline{y} = y_1,\dots,y_s$ we have:
    \begin{enumerate}[(1)]
        \item $(T_{n,\rho}(k)[\underline{x}]\llbracket\underline{y}
\rrbracket)^{1/p}$ is an intersection flat $T_{n,\rho}(k)[\underline{x}]\llbracket\underline{y}
\rrbracket$-algebra.
        \item $(T_n(k)^\dagger[\underline{x}]\llbracket\underline{y}\rrbracket)^{1/p}$ is an intersection flat $T_n(k)^\dagger[\underline{x}]\llbracket\underline{y}
\rrbracket$-algebra.
    \end{enumerate}
    Consequently, if $R$ is a reduced ring that is (essentially of) finite type over $T_{n,\rho}(k)$ or $T_n(k)^\dagger$ and $I$ is any ideal of $R$ (including $I = 0$), then the $I$-adic completion $\widehat{R}^I$ has big test elements.
\end{theorem}

% For the first case, we adapt the methods of \cite{DESTate},
% which establishes the intersection flatness of  $T_n(k)[\underline{x}]\llbracket\underline{y}\rrbracket^{1/p}$ over $T_n(k)[\underline{x}]\llbracket\underline{y}\rrbracket$, to the setting of \(T_{n,\rho}(k)[\underline{x}]\llbracket\underline{y}\rrbracket\). For the second case, we use a descent argument. Along the way, we give a detailed proof of the excellence of dagger algebras for the reader's convenience, since this fact is claimed without proof in \cite{GKDaggeralgebras}.

The Berkovich case adapts the intersection-flatness methods of \cite{DESTate} for classical Tate algebras, while the dagger case follows by descent. We also prove that dagger affinoid algebras are excellent, a fact previously asserted without proof in \cite{GKDaggeralgebras}.

\section{Preliminaries}

% For a ring $A$, an ideal $I$ and an $A$-module $M$, we always let $\widehat{M}^I$ denote the $I$-adic completion of $M$. If $A$ is Noetherian, then it is well-known that the canonical map $A \to \widehat{A}^I$ is flat. Moreover, if $\frp$ is a prime ideal of $A$ containing $I$, then $\frp \widehat{A}^I$ is a prime ideal of $\widehat{A}^I$ (containing $I\widehat{A}^I$) such that $\widehat{A}^I/\frp\widehat{A}^I \cong A/\frp$. Indeed, this follows upon applying $\otimes_A \widehat{A}^I$ to the exact sequence of finitely generated $A$-modules $0 \to \frp \to A \to A/\frp \to 0$, and noting that $A/\frp \otimes_A \widehat{A}^I \cong \widehat{A/\frp}^I \cong A/\frp$. Conversely, if $\frq \in \Spec(\widehat{A}^I)$ such that $\frq \supset I\widehat{A}^I$, then since the induced map $A/I \to \widehat{A}^I/I\widehat{A}^I$ is an isomorphism, it follows that $\frq$ is expanded from a prime ideal $\frp$ of $A$ that contains $I$, and hence, $\kappa(\frq) \cong \kappa(\frp)$.

For a ring \(A\), an ideal \(I\), and an \(A\)-module \(M\), let \(\widehat{M}^I\) denote the \(I\)-adic completion of \(M\). If \(A\) is Noetherian, the map \(A\to\widehat{A}^I\) is flat. If \(\mathfrak p\supseteq I\) is prime, then \(\widehat{A}^I/\mathfrak p\widehat{A}^I\cong\widehat{A/\mathfrak p}^{\,I}\cong A/\mathfrak p\), since \(I\) annihilates \(A/\mathfrak p\). Thus, \(\mathfrak p\widehat{A}^I\) is prime. Conversely, the isomorphism \(A/I\cong\widehat{A}^I/I\widehat{A}^I\) shows that every prime of \(\widehat{A}^I\) containing \(I\widehat{A}^I\) is the extension of a unique prime \(\mathfrak p\supseteq I\) of \(A\); corresponding primes have isomorphic residue fields.

\subsection{Regular ring maps and excellent rings}
\begin{definition}\cite[\href{https://stacks.math.columbia.edu/tag/07BZ}{Tag 07BZ}]{stacks-project}
    \label{def:regular-homomorphism}
    A ring homomophism $A \to B$ is \emph{regular} if the map is flat and if for all $\frp \in \Spec(A)$, the fiber $B \otimes_A \kappa(\frp)$ is a Noetherian geometrically regular $\kappa(\frp)$-algebra.
\end{definition}

We will frequently use Popescu's desingularization theorem for regular maps of Noetherian rings.

\begin{theorem}\cite{PopescuDesingularization, PopescuDesingularizationII, PopescuDesingularizationLetterEditor}
\label{thm:Popescu-desingularization}
Let $\varphi \colon A \to B$ be a flat map of Noetherian rings. Then $\varphi$ is regular if and only if $B$ can be expressed as a filtered colimit of smooth $A$-algebras.
\end{theorem}

\begin{definition}
\label{def:G-quasi-excellent}
Let $A$ be a Noetherian ring. We say $A$ is 
\begin{enumerate}
    \item a \emph{$G$-ring} or a \emph{Grothendieck ring} if for all $\frp \in \Spec(A)$, the canonical completion map $A_\frp \to \widehat{A_\frp}$ is regular.
    \item a \emph{quasi-excellent ring} if $A$ is a $G$-ring and if every finite type $A$-algebra has an open regular locus.
    \item an \emph{excellent ring} if $A$ is quasi-excellent and universally catenary.
\end{enumerate}
\end{definition}

\begin{remark}
\label{rem:CM-quasi-excellent}
Since Cohen-Macaulay rings are universally catenary \cite[\href{https://stacks.math.columbia.edu/tag/00NM}{Tag 00NM}]{stacks-project}, a quasi-excellent ring that is a homomorphic image of a Cohen-Macaulay ring is automatically excellent.
\end{remark}

The property of being geometrically regular can be tested at maximal ideals provided that the domain is a $G$-ring. This is known as the \textit{Grothendieck localization problem} for geometric regularity \cite{AndreGrothendieckLocalizationRegular}. The precise statement is:

\begin{theorem}\cite[Thm.\ on Pg 297]{AndreGrothendieckLocalizationRegular}
    \label{thm:Grothendieck-localization-regular}
    Let $\varphi \colon (A,\fram_A,\kappa_A) \to (B,\fram_B,\kappa_B)$ be a flat local homomorphism of Noetherian local rings such that $B/\fram_AB$ is a geometrically regular $\kappa_A$-algebra. If $A$ is a $G$-ring then $\varphi$ is a regular map.
\end{theorem}

A non-local corollary of the above result that will be useful for us is:

\begin{corollary}
\label{cor:Grothendieck-localization-regular-nonlocal}
Let $\varphi \colon A \to B$ be a flat map of Noetherian rings. Suppose for all maximal ideals $\fram$ of $B$, if $\frp \coloneqq \varphi^{-1}(\fram)$, then $B \otimes_A \kappa(\frp)$ is a geometrically regular $\kappa(\frp)$-algebra. If $A$ is a $G$-ring, then $\varphi$ is a regular map.
\end{corollary}

\begin{proof}
Since $B$ is Noetherian, by \cite[\href{https://stacks.math.columbia.edu/tag/07C0}{Tag 07C0, part (3)}]{stacks-project} it suffices to show that the induced local homomorphism $A_{\frp} \to B_{\fram}$ is regular. Note that $B_\fram \otimes_{A_\frp} \kappa(\frp)$ is a geometrically regular $\kappa(\frp)$-algebra because it is a further localization of $B \otimes_A \kappa(\frp)$. Thus, since $A_\frp$ is a $G$-ring, we get $A_{\frp} \to B_{\fram}$ is regular by \autoref{thm:Grothendieck-localization-regular}.
\end{proof}

We record an application of the previous corollary we will need later.

\begin{proposition}
\label{prop:completions-regular-maps}
Let $\varphi \colon A \to B$ be a regular map of Noetherian rings such that $A$ is quasi-excellent. If $I$ is an ideal of $A$, then the induced map on $I$-adic completion $\widehat{\varphi}^I \colon \widehat{A}^I \to \widehat{B}^{IB}$ is regular as well.
\end{proposition}

\begin{proof}
Since $B$ is $A$-flat, $\widehat{B}^{IB}$ is $\widehat{A}^I$-flat by \cite[\href{https://stacks.math.columbia.edu/tag/0AGW}{Tag 0AGW}]{stacks-project}. Since $A$ is quasi-excellent, $\widehat{A}^I$ is quasi-excellent hence a $G$-ring \cite{KuranoShimomotoIdealAdicCompletions}. Thus,  by \autoref{cor:Grothendieck-localization-regular-nonlocal} it suffices to show that if $\fram$ is a maximal ideal of $\widehat{B}^{IB}$ and $\frp \coloneqq \fram \cap \widehat{A}^I$, then $\widehat{B}^{IB} \otimes_{\widehat{A}^I} \kappa(\frp)$ is a geometrically regular $\kappa(\frp)$-algebra. Note that $I\widehat{B}^{IB}$ is contained in the Jacobson radical of $\widehat{B}^{IB}$ \cite[\href{https://stacks.math.columbia.edu/tag/05GI}{Tag 05GI}]{stacks-project}. Thus, $\fram \supset I\widehat{B}^{IB}$, and hence, $\frp \supset I\widehat{A}^I$. Consequently, $\frp$ must be expanded from a prime ideal $\frq$ of $A$ such that $\frq \supset I$ and $\kappa(\frp) \cong \kappa(\frq)$ (see discussion in the beginning of this section). Therefore, $\widehat{B}^{IB} \otimes_{\widehat{A}^I} \kappa(\frp) \cong (\widehat{B}^{IB} \otimes_{\widehat{A}^I} \widehat{A}^I/I\widehat{A}^I) \otimes_{\widehat{A}^I/I\widehat{A}^I} \kappa(\frp) \cong B/IB \otimes_{A/I} \kappa(\frq) \cong B \otimes_{A} \kappa(\frq)$ is a geometrically regular $\kappa(\frp) = \kappa(\frq)$-algebra by the assumption that $A \to B$ is a regular map. 
\end{proof}

\begin{remark}
\label{rem:faithful-flatness-completion}
If $\varphi \colon A \to B$ is a faithfully flat map of Noetherian rings and $I$ is an ideal of $A$, then $\widehat{\varphi}^I \colon \widehat{A}^I \to \widehat{B}^{IB}$ is faithfully flat as well. Indeed, as above, flatness of $\widehat{\varphi}^I$ follows by \cite[\href{https://stacks.math.columbia.edu/tag/0AGW}{Tag 0AGW}]{stacks-project}. It suffices to show that if $\fram$ is a maximal ideal of $\widehat{A}^I$, then $\fram\widehat{B}^{IB} \neq \widehat{B}^{IB}$. But $I\widehat{A}^I$ is contained in the Jacobson radical of $\widehat{A}^I$, so $\fram = \fram'\widehat{A}^I$, where $\fram'$ is a maximal ideal of $A$ such that $\fram' \supset I$. Let $\frq \in \Spec(B)$ lie over $\fram'$. Then $\frq \supset IB$, and so, $\frq\widehat{B}^{IB}$ is a prime ideal of $\widehat{B}^{IB}$. By construction, $\fram\widehat{B}^{IB} = \fram'\widehat{B}^{IB} \subset \frq\widehat{B}^{IB}$, proving the desired assertion.
\end{remark}

% \begin{remark}
%     \label{rem:weakning-hypothesis-G-ring}
%     We do not know if the conclusion of \autoref{prop:completions-regular-maps} holds if we assume that $A$ is only a $G$-ring. Our proof relies on $\widehat{A}^I$ being a $G$-ring, which follows if $A$ is quasi-excellent. However, the ideal adic completion of a $G$-ring need not be a $G$-ring \cite{NishimuraIdealAdicCompletions}.
% \end{remark}

We will need an analog of \cite[\href{https://stacks.math.columbia.edu/tag/0GIM}{Tag 0GIM}]{stacks-project} for filtered colimits of smooth ring maps.

\begin{lemma}
\label{lem:colimit-indsmooth-is-indsmooth}
Let $(I,\leq)$ be a directed poset. Let $I \ni i \mapsto (R_i \to A_i)$ be a system of arrows of rings over $I$. Let $R = \colim R_i$ and $A = \colim A_i$. If each $A_i$ is a filtered colimit of smooth $R_i$-algebras, then $A$ is a filtered colimit of smooth $R$-algebras.
\end{lemma}

\begin{proof}
Since smoothness is preserved under arbitrary base change, for all $i$, $R \otimes_{R_i} A_i$ is a filtered colimit of smooth $R$-algebras. Now, 
\[A = \colim_i R \otimes_{R_i} A_i.\]
Hence we have reduced the lemma to showing that if $R$ is a ring and $A_i$ is a system of rings indexed by $I$ such that each $A_i$ is a filtered colimit of smooth $R$-algebras, then $\colim A_i$ is a filtered colimit of smooth algebras. Since smooth maps are of finite presentation one now adapts the proof of \cite[\href{https://stacks.math.columbia.edu/tag/0BSJ}{Tag 0BSJ}]{stacks-project} to the smooth setting.
\end{proof}

Using Popescu's desingularization result one then readily obtains:

\begin{corollary}
\label{cor:colimit-regular-is-regular}
Let $(I,\leq)$ be a directed poset. Let $I \ni i \mapsto (R_i \to A_i)$ be a system of homomorphisms of Noetherian rings over $I$. Let $R = \colim R_i$ and $A = \colim A_i$. Suppose $R$ and $A$ are Noetherian. If each $R_i \to A_i$ is regular then the induced map $R \to A$ is regular.
\end{corollary}

We record a descent result for the (quasi)-excellence property. A proof is included for the reader's convenience.

\begin{lemma}
\label{lem:descent-quasi-excellence}
Let $\varphi \colon A \to B$ be a faithfully flat regular map of rings. If $B$ is quasi-excellent, then $A$ is quasi-excellent. If $B$ is quasi-excellent and Cohen-Macaulay, then $A$ is excellent.
\end{lemma}

\begin{proof}
The property of being Noetherian satisfies faithfully flat descent. Thus, quasi-excellence of $B$ implies $A$ is Noetherian. We first show that the property of being a $G$-ring satisfies faithfully flat descent. Let $\frp \in \Spec(A)$ and $\frq \in \Spec(B)$ such that $\frq \cap A = \frp$. The induced map $A_\frp \to B_\frq$ is regular by \cite[\href{https://stacks.math.columbia.edu/tag/07C0}{Tag 07C0}]{stacks-project}. Then consider the commutative diagram
\[\begin{tikzcd}[cramped]
	{A_\frp} & {B_\frq} \\
	{\widehat{A_\frp}} & {\widehat{B_\frq}}
	\arrow[from=1-1, to=1-2]
	\arrow[from=1-1, to=2-1]
	\arrow[from=1-2, to=2-2]
	\arrow[from=2-1, to=2-2],
\end{tikzcd}\]
where the bottom horizontal arrow is the induced map on completions. If $B$ is a $G$-ring then $A_\frp \to \widehat{B_\frq}$ is a regular ring map by \cite[\href{https://stacks.math.columbia.edu/tag/07QI}{Tag 07QI}]{stacks-project}. Since $\widehat{A_\frp} \to \widehat{B_\frq}$ is faithfully flat, $A_\frp \to \widehat{A_\frp}$ must be regular by \cite[\href{https://stacks.math.columbia.edu/tag/07NT}{Tag 07NT}]{stacks-project}. Thus, $A$ is a $G$-ring.

If $\varphi \colon A \to B$ is a flat regular map of Noetherian rings, then $\Reg(B) = \Spec(\varphi)^{-1}(\Reg(A))$. Indeed, by faithfully flat descent of regularity, $\Reg(B) \subset \Spec(\varphi)^{-1}(\Reg(A))$. Conversely, let $\frq \in \Spec(\varphi)^{-1}(\Reg(A))$. Then  $A_{\varphi^{-1}(\frq)} \to B_\frq$ is a regular map and $A_{\varphi^{-1}(\frq)}$ is regular, so $B_\frq$ is also regular by \cite[\href{https://stacks.math.columbia.edu/tag/0H7S}{Tag 0H7S}]{stacks-project}. Thus, $\Spec(\varphi)^{-1}(\Reg(A)) \subset \Reg(B)$.

Let $R$ be a finite type $A$-algebra. We want to show $\Reg(R)$ is open. By finite type base change, $\id_R \otimes_A \varphi \colon R \to R \otimes_A B$ is also regular \cite[\href{https://stacks.math.columbia.edu/tag/07C1}{Tag 07C1}]{stacks-project}, and it is faithfully flat. The latter implies by \cite[\href{https://stacks.math.columbia.edu/tag/02JY}{Tag 02JY}]{stacks-project} that a subset $T$ of $\Spec(R)$ is open if and only if $\Spec(\id_R \otimes_A \varphi)^{-1}(T)$ is open in $\Spec(R \otimes_A B)$. Thus, it suffices to show that $\Reg(R \otimes_A B) = \Spec(\id_R \otimes_A \varphi)^{-1}(\Reg(R))$ is open. But this holds because $B$ is quasi-excellent and $R \otimes_A B$ is a finite type $B$-algebra.

%Now $R \to R \otimes_A B$ is a faithfully flat and finitely presented ring map. Thus, the induced map on $\Spec$ is open \cite[\href{https://stacks.math.columbia.edu/tag/00I1}{Tag 00I1}]{stacks-project}. Since $B$ is quasi-excellent and $R \otimes_A B$ is a finite type $B$-algebra, $\Reg(R \otimes_A B)$ is open. Thus, $\Reg(R)$ is open by the previous paragraph since it is the image of the open set $\Reg(R \otimes_A B)$ under the induced map on $\Spec$.

Finally, if $B$ is quasi-excellent and Cohen-Macaulay, then $A$ is quasi-excellent by the above and Cohen-Macaulay by \cite[\href{https://stacks.math.columbia.edu/tag/0352}{Tag 0352}]{stacks-project}. Thus, $A$ is excellent by \autoref{rem:CM-quasi-excellent}.
\end{proof}

\begin{remark}
\label{rem:universally-caternary-does-not-descend}
The property of being universally catenary does not descend along faithfully flat maps in general. In fact, this property does not descend even along an {\'e}tale ring map that is surjective on $\Spec$ \cite[\href{https://stacks.math.columbia.edu/tag/0355}{Tag 0355}]{stacks-project}.
\end{remark}

\subsection{Tate algebras and generalizations}
Recall that a \emph{non-Archimedean norm} on a field $k$ is a function
$
|\cdot| \colon k \to \mathbb{R}_{\geq 0}    
$
that satisfies the following properties:
\begin{enumerate}
    \item[$\bullet$] $|x| = 0$ if and only if $x = 0$,
    \item[$\bullet$] $|xy| = |x||y|$, and
    \item[$\bullet$] $|x + y| \leq \max\{|x|,|y|\}$.
\end{enumerate}
In other words, $|\cdot|$ is a non-Archimedean multiplicative valuation of $k$ of rank $1$.
A field $k$ equipped with a non-Archimedean norm $|\cdot|$ is called a \emph{real-valued field} and is denoted $(k,|\cdot|)$.

The valuation ring of $k$ is the subring $k^\circ \coloneqq \{x \in k \colon |x| \leq 1\}$. This is a local ring with maximal ideal $k^{\circ \circ} \coloneqq \{x \in k \colon |x| < 1\}$. Note that $k^\circ$ has Krull dimension $1$.

\begin{definition}
    \label{def:NA-field}
    A \emph{non-Archimedean} (abbrev. NA) field is a real-valued field $(k,|\cdot|)$ such that $k$ is complete with respect to the metric $|x - y|$ that is induced by $|\cdot|$ and such that $|k^\times| \neq \{1\}$, that is, $k$ is non-trivially valued. Here $|k^\times|$ denotes the \textit{value group}, i.e. the image under $|\cdot|$ of the set of non-zero elements, $k^\times$, of $k$.
\end{definition}

\begin{definition}
    \label{def:Tate-algebra}
    Let $(k,|\cdot|)$ be a NA field. For every positive integer $n$, the \emph{Tate algebra} in $n$ indeterminates over $k$, denoted, $T_n(k)$, is the $k$-subalgebra of the formal power series ring $k\llbracket X_1,\dots,X_n\rrbracket $ consisting of those power series
    $
    \sum_{\nu \in \mathbb{Z}_{\geq 0}^n} a_\nu X^{\nu}
    $
    (written in multi-index notation) such that $|a_\nu| \to 0$ as $\nu_1 + \dots + \nu_n \to \infty$. An element of $T_n(k)$ is called a \emph{restricted power series}.  A homomorphic image of $T_n(k)$ is called a \emph{strictly affinoid $k$-algebra} or \emph{strictly $k$-affinoid}.
\end{definition}

\begin{remark}
What we call ``strictly affinoid $k$-algebras'' are often simply called \emph{affinoid $k$-algebras}. The term ``strictly'' tends to be used in the literature on Berkovich spaces \cite{BerkovichSpectralTheory,BerkovichEtale}. On the other hand, when dealing with other theories of $p$-adic analytic geometry, like Huber's adic spaces,  it is more common to use the terminology "affinoid $k$-algebra". Since affinoid $k$-algebras in the sense of Berkovich will play a prominent role in this paper, we will follow Berkovich's practice of calling homomorphic images of $T_n(k)$ strictly $k$-affinoid.
\end{remark}

The Tate algebra is a $k$-Banach algebra (i.e. a complete normed $k$-algebra) with respect to the \emph{Gauss norm}, which is defined as follows: for all $\sum_{\nu \in \mathbb{Z}_{\geq 0}^n} a_\nu X^{\nu} \in T_n(k)$, 
\[
\bigg|\bigg|\sum_{\nu \in \mathbb{Z}_{\geq 0}^n} a_\nu X^{\nu}\bigg|\bigg|   \coloneqq \max\{|a_\nu| \colon \nu \in \mathbb{Z}_{\geq 0}^n\}.
\]
The Gauss norm on $T_n(k)$ is multiplicative \cite[pp. 13-14]{Boschrigid} (see also \cite[Corollary~2.24]{garzella-2024-perfectoid-tate-uncountable} for a generalization). We will always consider $T_n(k)$ as a $k$-Banach algebra with respect to the Gauss norm.

Remarkably, $T_n(k)$ shares many of the properties of the polynomial ring $k[X_1,\dots,X_n]$. We summarize some of these properties below for the reader's convenience.

\begin{theorem}
    \label{thm:Tate-algebras-properties}
    Let $(k,|\cdot|)$ be a NA field, and let $n$ be a positive integer. Then the Tate algebra $T_n(k)$ satisfies the following properties:
    \begin{enumerate}%[label=\textnormal{(\alph*)}]
        \item \label{thm:Tate-Noetherian} $T_n(k)$ is Noetherian.
        \item \label{thm:Tate-UFD} $T_n(k)$ is a unique factorization domain.
        \item \label{thm:Tate-Jacobson} $T_n(k)$ is Jacobson, that is, every radical ideal is the intersection of the maximal ideals containing it.
        \item \label{thm:Tate-maximal-ideal-expanded} Let $k[X_1,\dots,X_n] \hookrightarrow T_n(k)$ be the inclusion map. If $\fram$ is a maximal ideal of $T_n(k)$, then $\fram' \coloneqq \fram \cap k[X_1,\dots,X_n]$ is a maximal ideal of $k[X_1,\dots,X_n]$. Furthermore, $\fram'T_n(k) = \fram$ and the induced map on residue fields $k[X_1,\dots,X_n]/\fram' \hookrightarrow T_n(k)/\fram$ is an isomorphism.
        \item \label{thm:Tate-dimension} If $\fram$ is a maximal ideal of $T_n(k)$, then $\fram$ is generated by $n$ elements and $T_n(k)/\fram$ is a finite extension of $k$. Moreover, $\fram$ has height $n$ and hence $T_n(k)$ has Krull dimension $n$.
        \item \label{thm:Tate-regular} $T_n(k)$ is regular.
        \item \label{thm:Tate-excellent} $T_n(k)$ is excellent, and hence, all strictly affinoid $k$-algebras are excellent.
        \item \label{thm:Tate-ideals-closed} Every ideal of $T_n(k)$ is closed in the topology on $T_n(k)$ that is induced by the Gauss norm.
        \item \label{thm:Tate-finite-extension} If $T_n(k) \to A$ is a finite ring map, then $A$ is a strictly affinoid $k$-algebra.
    \end{enumerate}
\end{theorem}

\begin{proof}[Indication of proof]
    \autoref{thm:Tate-Noetherian}--\autoref{thm:Tate-Jacobson} are proved in \cite[Propositions 2.2/14--16]{Boschrigid}, while \autoref{thm:Tate-maximal-ideal-expanded} follows from \cite[7.1/2]{BGRNAanalysis}. The first assertion of \autoref{thm:Tate-dimension} follows from \autoref{thm:Tate-maximal-ideal-expanded} and the corresponding fact for polynomial rings over a field, while the second assertion about the height of each maximal ideal follows by \cite[2.2/17]{Boschrigid}. \autoref{thm:Tate-regular} follows from \autoref{thm:Tate-dimension}. The most difficult property to prove is \autoref{thm:Tate-excellent}, which is shown in \cite[Theorem\ 3.3]{KiehlTate}. \autoref{thm:Tate-ideals-closed} follows by \cite[Corollary\ 2.3/8]{Boschrigid}. Finally, \autoref{thm:Tate-finite-extension} follows by \cite[6.1.1/5]{BGRNAanalysis} because every homomorphism of (strictly) affinoid algebras is continuous by \cite[6.1.3/1]{BGRNAanalysis}.
\end{proof}

\begin{corollary}
\label{cor:polynomial-to-Tate-regular}
Let $(k,|\cdot|)$ be a NA field. Then the inclusion $i \colon k[X_1,\dots,X_n] \hookrightarrow T_n(k)$ is a regular map. Furthermore, for all maximal ideals $\fram$ of $T_n(k)$, the map on completions induced by the local homomorphism $k[X_1,\dots,X_n]_{\fram\cap k[X_1,\dots,X_n]} \to T_n(k)_\fram$ is an isomorphism.
\end{corollary}

\begin{proof}
If we can show that $i$ is flat, then its regularity will follow by \autoref{thm:Tate-algebras-properties}\autoref{thm:Tate-maximal-ideal-expanded} and \autoref{cor:Grothendieck-localization-regular-nonlocal}. But flatness is a local property that can be checked after localizing at the maximal ideals of the target \cite[\href{https://stacks.math.columbia.edu/tag/00HT}{Tag 00HT}]{stacks-project}, and for a local homomorphism of Noetherian local rings flatness can be checked after completion \cite[\href{https://stacks.math.columbia.edu/tag/0C4G}{Tag 0C4G}]{stacks-project}. If $\fram$ is a maximal ideal of $T_n(k)$, then the map on completions induced by  $k[X_1,\dots,X_n]_{\fram \cap k[X_1,\dots,X_n]} \to T_n(k)_{\fram}$ is surjective by \cite[\href{https://stacks.math.columbia.edu/tag/0315}{Tag 0315, part (1)}]{stacks-project} and \autoref{thm:Tate-algebras-properties}\autoref{thm:Tate-maximal-ideal-expanded}. But the induced map on completions is a homomorphism of regular local rings of Krull dimension $n$ by \autoref{thm:Tate-algebras-properties}\autoref{thm:Tate-dimension}. Thus, it must be an isomorphism by dimension considerations and is hence flat.
\end{proof}

\begin{remark}
$k[X_1,\dots,X_n] \hookrightarrow T_n(k)$ is not faithfully flat. For example, if $k$ is algebraically closed, then the maximal spectrum of $k[X_1,\dots,X_n]$ is in bijection with $k^n$ whereas the maximal spectrum of $T_n(k)$ is in bijection with the unit polydisc $\mathbb{B}^n = \{(\alpha_1,\dots,\alpha_n) \in k^n \colon |\alpha_i| \leq 1\}$ \cite[7.1.1/1]{BGRNAanalysis}.
\end{remark}

Let $\overline{k}$ be an algebraic closure of a NA field $(k,|\cdot |)$. Then there is a unique extension of the norm on $k$ to $\overline{k}$ \cite[Appendix~A, Theorem~3]{Boschrigid}. If we denote this extension also by $|\cdot|$, then the starting point of rigid geometry is to interpret $T_n(k)$ as the ring of functions on the closed unit polydisc
\[\mathbb{B}^n(\overline{k}) \coloneqq \{(\alpha_1,\cdots,\alpha_n) \in \overline{k}^n \colon |\alpha_i| \leq 1 \text{ for all } i\}.\]
But say now that for $\rho \coloneqq (\rho_1,\dots,\rho_n) \in (\mathbb{R}_{> 0})^n$ one is interested in defining the ring of convergent power series on the general polydisc
\[\mathbb{B}^n_\rho(\overline{k}) \coloneqq \{(\alpha_1,\cdots,\alpha_n) \in \overline{k}^n \colon |\alpha_i| \leq \rho_i \text{ for all } i\}.\]
This leads to Berkovich's notion of affinoid algebras.

\begin{definition}
\label{def:convergent-power-series-general-polydisc}
For $\rho \coloneqq (\rho_1,\dots,\rho_n) \in (\mathbb{R}_{> 0})^n$, one defines the ring of \emph{convergent power series in the polydisc $\mathbb{B}^n_\rho(\overline{k})$}, denoted $T_{n,\rho}(k)$, to be the ring  consisting of those power series
    $
    \sum_{\nu \in \mathbb{Z}_{\geq 0}^n} a_\nu X^{\nu}
    $
    (written in multi-index notation) such that $|a_\nu|\rho^\nu \coloneqq |a_\nu|\rho_1^{\nu_1}\cdots\rho_n^{\nu_n} \to 0$ as $\nu_1 + \dots + \nu_n \to \infty$. A homomorphic image of $T_{n,\rho}(k)$ is called an \emph{affinoid $k$-algebra in the sense of Berkovich}, or a \emph{Berkovich affinoid $k$-algebra} for short.
\end{definition}

Thus, for $\rho = (1,\dots,1)$, $T_{n,\rho}(k) = T_n(k)$. Note that if all the coordinates of $\rho$ are $\geq 1$, then $T_{n,\rho}(k)$ is a $k[X_1,\dots,X_n]$-subalgebra of $T_n(k)$. Generalizing the Gauss norm on $T_n(k)$, we set the \emph{Gauss norm} on $T_{n,\rho}(k)$ to be the one where for $\sum_{\nu \in \mathbb{Z}_{\geq 0}^n} a_\nu X^{\nu} \in T_{n,\rho}(k)$, 
\[\left|\left|\sum_{\nu \in \mathbb{Z}_{\geq 0}^n} a_\nu X^{\nu}\right|\right| \coloneqq \max\{ |a_\nu|\rho^{\nu} \colon \nu \in \mathbb{Z}^n_{\geq 0}\}.\]
The Gauss norm on $T_{n,\rho}(k)$ is multiplicative \cite[6.1.5/2]{BGRNAanalysis} and makes $T_{n,\rho}(k)$ into a $k$-Banach algebra that contains $k[X_1, \dots, X_n]$ as a dense $k$-subalgebra \cite[6.1.5/1]{BGRNAanalysis}. 

%Thus, $T_{n,\rho}(k)$ is the ring of functions on the polydisk $\mathbb{B}^n_\rho(\overline{k})$.

We now summarize some of the main properties of $T_{n,\rho}(k)$ that we need.

\begin{proposition}
\label{prop:convergent-powerseries-general-polydisc}
Let $(k,|\cdot|)$ be a NA field. Fix an algebraic closure $\overline{k}$ of $k$ and let $|\cdot|$ also denote the unique norm on $\overline{k}$ extending the one on $k$. Let $\rho = (\rho_1,\dots,\rho_n) \in (\mathbb{R}_{> 0})^n$, and let $\left|\overline{k}^\times\right|^n_{> 0}$ denote those elements of $(\mathbb{R}_{> 0})^n$ all of whose coordinates are in $\left|\overline{k}^\times\right|$. We have the following:
\begin{enumerate}
    \item\label{prop:convergent-powerseries-general-polydisc.a} $T_{n,\rho}(k)$ is a strictly affinoid $k$-algebra if and only if $\rho \in \left|\overline{k}^\times\right|^n_{> 0}$. In this case all the maximal ideals of $T_{n,\rho}(k)$ have height $n$.
    \item\label{prop:convergent-powerseries-general-polydisc.b} If $(k,|\cdot|_k) \hookrightarrow (\ell,|\cdot|_\ell)$ is an extension of NA fields, then $T_{n,\rho}(k) \widehat{\otimes_k} \ell \cong T_{n,\rho}(\ell)$ and $T_{n,\rho}(k) \hookrightarrow T_{n,\rho}(\ell)$ is faithfully flat.
    \item\label{prop:convergent-powerseries-general-polydisc.c} There exists an extension of NA fields $(k,|\cdot|_k) \hookrightarrow (\ell,|\cdot|_\ell)$ such that $T_{n,\rho}(\ell)$ is strictly $\ell$-affinoid.
    \item \label{prop:convergent-powerseries-general-polydisc.d} $T_{n,\rho}(k)$ is excellent and regular.
    %\item \label{prop:convergent-powerseries-general-polydisc.e} It is regular and every maximal ideal has height $n$.
    \item \label{prop:convergent-powerseries-general-polydisc.e} If all the $\rho_i \geq 1$, then the inclusion $T_{n,\rho}(k) \hookrightarrow T_n(k)$ is a regular map.
    \item \label{prop:convergent-powerseries-general-polydisc.f} Every ideal of $T_{n,\rho}(k)$ is closed in the topology induced by the Gauss norm.
\end{enumerate}
\end{proposition}

\begin{proof}[Indication of proof]
\autoref{prop:convergent-powerseries-general-polydisc.a} This is \cite[6.1.5/4]{BGRNAanalysis}. We summarize the argument for the backwards implication since we will use it in the later parts. If $\rho \in |k^\times|^n_{> 0}$, then choosing $c_i \in k$ such that $|c_i| = \rho_i$ one sees that the map 
\begin{align*}
T_n(k) &\longrightarrow T_{n,\rho}(k)\\
X_i &\longmapsto c^{-1}_iX_i
\end{align*}
is a norm-preserving (i.e. isometric) isomorphism. In the general case, one can define a finite injective ring map $T_n(k) \hookrightarrow T_{n,\rho}(k)$  and then $T_{n,\rho}(k)$ is strictly $k$-affinoid by \autoref{thm:Tate-algebras-properties}\autoref{thm:Tate-finite-extension}. Since $T_n(k)$ is a normal domain, the height of a maximal ideal of $T_{n,\rho}(k)$ must equal the height of the maximal ideal of $T_n(k)$ it lies over, and the latter height is $n$ by \autoref{thm:Tate-algebras-properties}\autoref{thm:Tate-dimension}.

\autoref{prop:convergent-powerseries-general-polydisc.b} The isomorphism $T_{n,\rho}(k) \widehat{\otimes_k} \ell \cong T_{n,\rho}(\ell)$ is clear and faithful flatness of $T_{n,\rho}(k) \hookrightarrow T_{n,\rho}(\ell)$ follows by this isomorphism and \cite[Lemma~2.1.2]{BerkovichEtale}.

\autoref{prop:convergent-powerseries-general-polydisc.c} follows by the discussion in \cite[Page 22]{BerkovichSpectralTheory}.

\autoref{prop:convergent-powerseries-general-polydisc.d} Excellence is again difficult to show and follows by \cite[Th\'eor\`eme~2.6]{DucrosBerkovichExcellent}. We include a proof of regularity for the reader's convenience. By faithfully flat descent of regularity \cite[\href{https://stacks.math.columbia.edu/tag/07NG}{Tag 07NG}]{stacks-project}, \autoref{prop:convergent-powerseries-general-polydisc.c} and \autoref{prop:convergent-powerseries-general-polydisc.b}, we may assume $T_{n,\rho}(k)$ is strictly affinoid. By \autoref{prop:convergent-powerseries-general-polydisc.a} we can choose a finite field extension $\ell$ of $k$ contained in $\overline{k}$ so that $\rho \in |\ell^\times|^n_{> 0}$. Again, replacing $k$ by $\ell$ we may assume $\rho \in |k^\times|^n_{>0}$. Then by the sketch in \autoref{prop:convergent-powerseries-general-polydisc.a} and \autoref{thm:Tate-algebras-properties}\autoref{thm:Tate-regular}, $T_{n,\rho}(k) \cong T_n(k)$ is regular.

\autoref{prop:convergent-powerseries-general-polydisc.e} Since this property will be crucial for us later, we provide a detailed proof. It is clear that if all the $\rho_i \geq 1$, then $T_{n,\rho}(k)$ is a $k[X_1,\dots,X_n]$ subalgebra of $T_n(k)$. We first show $T_{n,\rho}(k) \hookrightarrow T_n(k)$ is flat. For this, let $\fram$ be a maximal ideal of $T_n(k)$. Let $\eta \coloneqq T_{n,\rho}(k) \cap \fram$ and $\fram' \coloneqq k[X_1,\dots,X_n] \cap \fram$. We need $T_{n,\rho}(k)_\eta \to T_n(k)_\fram$ to be flat. Since the composition 
\[k[X_1,\dots,X_n]/\fram' \hookrightarrow T_{n,\rho}(k)/\eta \hookrightarrow T_n(k)/\fram\]
is an isomorphism (\autoref{thm:Tate-algebras-properties}\autoref{thm:Tate-maximal-ideal-expanded}), $\eta$ is maximal and $T_{n,\rho}(k)/\eta \cong T_n(k)/\fram$. Moreover, since $\fram'T_n(k) = \fram$, we have $\eta T_n(k) = \fram$ as well. Thus, the induced map on completions 
\[\widehat{T_{n,\rho}(k)_{\eta}} \to \widehat{T_n(k)_\fram}\]
is again surjective by \cite[\href{https://stacks.math.columbia.edu/tag/0315}{Tag 0315, part (1)}]{stacks-project}. Since $\widehat{T_{n,\rho}(k)_{\eta}}$ and $\widehat{T_n(k)_\fram}$ are regular local rings, this shows that $\text{height}(\eta) \geq \height(\fram) = n$, where the last equality follows by \autoref{thm:Tate-algebras-properties}\autoref{thm:Tate-dimension}. We claim that $\text{height}(\eta) = n$, whence $\widehat{T_{n,\rho}(k)_{\eta}} \to \widehat{T_n(k)_\fram}$ is an isomorphism (and hence flat) by a Krull dimension consideration.

To see our claim, if we choose a faithfully flat extension $T_{n,\rho}(k) \hookrightarrow T_{n,\rho}(\ell)$ so that $T_{n,\rho}(\ell)$ is strictly $\ell$-affinoid, then for any maximal ideal $\eta'$ of $T_{n,\rho}(\ell)$ that lies over $\eta$, we have $\text{height}(\eta) \leq \height(\eta') = n$ by going down for flat maps \cite[\href{https://stacks.math.columbia.edu/tag/00HS}{Tag 00HS}]{stacks-project} and \autoref{prop:convergent-powerseries-general-polydisc.a}.  

That $T_{n,\rho}(k) \hookrightarrow T_n(k)$ is regular now follows by \autoref{cor:Grothendieck-localization-regular-nonlocal} because the fiber ring $T_n(k) \otimes_{T_{n,\rho}(k)} \kappa(\eta) \cong T_n(k)/\fram$ is isomorphic to $\kappa(\eta) = T_{n,\rho}(k)/\eta$.

\autoref{prop:convergent-powerseries-general-polydisc.f} Using \autoref{prop:convergent-powerseries-general-polydisc.c}, the proof of this reduces to the case of Tate algebras; see \cite[Proposition~2.1.3]{BerkovichSpectralTheory}.
\end{proof}

\begin{remark}
    Unlike the case of strictly affinoid $k$-algebras, Noether normalization does not hold for Berkovich affinoid $k$-algebras on the nose. See \cite[Remark~ 3.1.2.4 (iii)]{TemkinBerkovich}.
\end{remark}

\subsection{The relative Frobenius} If $\varphi \colon R \to S$ is a homomorphism
of rings of prime characteristic $p > 0$, then for every integer $e \geq 0$, consider the co-Cartesian diagram
  \[
    \begin{tikzcd}[column sep=4em]
      R \rar{F_R^e}\dar[swap]{\varphi} & F^e_{R*}R
      \arrow[bend left=30]{ddr}{F^e_{R*}\varphi}
      \dar{\id_{F^e_{R*}R} \otimes_R \varphi }\\
      S \rar{F^e_R \otimes_R \id_S} \arrow[bend right=12,end
      anchor=west]{drr}[swap]{F^e_S} & F^e_{R*}R \otimes_R S
      \arrow[dashed]{dr}[description]{F^e_{S/R}}\\
      & & F^e_{S*}S,
    \end{tikzcd}
    \]
    where $F^e_R \colon R \to F^e_{R*}R$ (resp. $F^e_S \colon S \to F^e_{S*}S$)
    denotes the $e$-th iterate of the Frobenius endomorphism on $R$ (resp. on $S$). 
    The \emph{$e$-th relative Frobenius of $\varphi$} is the ring homomorphism
  \[
    \begin{tikzcd}[column sep=1.475em,row sep=0]
      \mathllap{F^e_{S/R}\colon} F^e_{R*}R \otimes_R S \rar & F^e_{S*}S\\
      F^e_{R*}r \otimes s \rar[mapsto] & F^e_{S*}\varphi(r)s^{p^e}
    \end{tikzcd}
  \]
Here if $r \in R$, we denote the corresponding element of $F^e_{R*}R$ by $F^e_{R*}r$. If $e = 1$, we denote $F^1_{S/R}$ by $F_{S/R}$.

\begin{notation}
From now on, for a ring $R$ of prime characteristic $p > 0$, by $F^e_*R$ we will always mean $F^e_{R*} R$ for ease of notation. That is, we will drop the $R$ in the subscript.
\end{notation}

The relative Frobenius map $F_{S/R}$ detects geometric properties of the fibers of
$\varphi \colon R \to S$ when $R$ and $S$ are Noetherian. This is summarized in the 
next result.

\begin{theorem}
\label{thm:Radu-Andre-Dumitrescu}
Let $\varphi \colon R \to S$ be a flat homomorphism of Noetherian rings of 
prime characteristic $p > 0$. Then we have the following:
\begin{enumerate}[label=\textnormal{(\arabic*)}]
    \item $\varphi$ is a regular map if and only if $F_{S/R}$ is flat.
    \label{thm:Radu-Andre-Dumitrescu.1}
    \item The fibers of $\varphi$ are geometrically reduced if and only if $F_{S/R}$ is pure as a map of $F_*R$-modules.
    \label{thm:Radu-Andre-Dumitrescu.2}
\end{enumerate}
\end{theorem}

\begin{proof}[Indication of proof]
   \ref{thm:Radu-Andre-Dumitrescu.1} follows by \cite[Theorem\ 4]{RaduUneClasseDAnneaux} and
    \cite[Theorem\ 1]{AndreHomomorphismsRegulariers} while \ref{thm:Radu-Andre-Dumitrescu.2} follows by 
    \cite[Theorem\ 3]{DumitrescuReduceness}.
\end{proof}

\subsection{\emph{F}-intersection flatness and Ohm-Rush trace}

\begin{definition}
    \label{def:IF}
    For $R$-modules $L$ and $M$ and a submodule $U \subseteq L$, we denote by $UM$ the image of $U \otimes_R M$ in $L \otimes_R M$. 
    An $R$-module $M$ is called \emph{intersection flat} if for any finitely generated $R$-module $L$ and any collection of submodules $\{U_i\}_i$ of $L$, we have that 
    \[
        \Big( \bigcap_i U_i \Big) M = \bigcap_i \big( U_i M \big).
    \]
\end{definition}

Note that an intersection flat module is in particular flat by \cite[Proposition 5.5]{HochsterJeffriesintflatness}. $F$-intersection flatness is a special case of intersection flatness.

\begin{definition}
\label{def:FIF}
Let $R$ be a ring of prime characteristic $p > 0$. Then $R$ is \emph{$F$-intersection flat} if $F_*R$ is an intersection flat $R$-module.
\end{definition}

Thus, if $R$ is Noetherian and $F$-intersection flat, then $R$ is regular by \cite{KunzCharacterizationsOfRegularLocalRings}. 

We record a descent result for $F$-intersection flatness that relies on \autoref{thm:Radu-Andre-Dumitrescu}.

\begin{proposition}
\label{prop:descent-FIF}
Let $R \to S$ be a faithfully flat regular map of Noetherian rings of prime characteristic $p > 0$. If $S$ is $F$-intersection flat, then $R$ is $F$-intersection flat.
\end{proposition}

\begin{proof}
$F_{S/R}$ is a faithfully flat ring map by \autoref{thm:Radu-Andre-Dumitrescu}, and hence, it is pure as a map of $S$-modules. Then the result follows by \cite[Theorem~3.4.1, part (1)]{DESTate}.
\end{proof}

A less mysterious notion than intersection flatness is given by the so-called Ohm-Rush trace property. For an $R$-module $M$ and $x \in M$, by \emph{the trace of $x$}, denoted $\Tr_M(x)$ or $\Tr(x)$ (when $M$ is clear from context), we mean the ideal of $R$ given by the image of the evaluation at $x$ map:
\[\im\left(\Hom_R(M,R) \xrightarrow{\text{ev}\MVAt x} R\right).\]

\begin{definition}
    \label{def:Ohm-Rush-trace}
    An $R$-module $M$ is \emph{Ohm-Rush trace} if for all $x \in M$, $x \in \Tr_M(x)M$. A ring homomorphism $R \to S$ is \emph{Ohm-Rush trace} if $S$ is an Ohm-Rush trace $R$-module.
\end{definition}

Ohm-Rush trace modules appear to have been first systematically studied in  \cite{OhmRushcontent} where they were called \emph{trace modules}. However, since the term `trace module' can mean many different things in mathematics, we choose this slightly more verbose terminology following \cite{DattaEpsteinTuckerML}.

Ohm-Rush trace modules are intersection flat \cite[Proposition~4.3.8]{DattaEpsteinTuckerML}, although for modules over Noetherian complete local rings, the two notions coincide \cite[Theorem~4.3.12]{DattaEpsteinTuckerML}.

\begin{definition}
    \label{def:Frobenius-Ohm-Rush-trace}
    A ring $R$ of prime characteristic $p > 0$ is \emph{Frobenius Ohm-Rush trace}, abbreviated \emph{FORT}, if $F_*R$ is an Ohm-Rush trace $R$-module.
\end{definition}

\begin{example}
Let $R$ be a regular ring that is essentially of finite type over a local $G$-ring $(A,\fram)$ of prime characteristic $p > 0$. Then $R$ is $F$-intersection flat. Moreover, if $A$ is $\fram$-adically complete, then $R$ is Frobenius Ohm-Rush trace. For both results see the proof of \cite[Theorem~3.4.1 (3)]{DESTate} and the references therein.
\end{example}

For the sake of completeness, we include the definition of big test elements and the main result (for us) on existence of big test elements. However, we are omitting the definition of tight closure.

\begin{definition}
    \label{def:big-test-elements}
    For a Noetherian ring $R$ of prime characteristic $p > 0$, an element $c \in R$ not in any minimal prime is called a \emph{big test element of $R$} if for all $R$-modules $M$ and for all submodules $N$ of $M$, we have $c(N^*_M) \subset N$, where $N^*_M$ is the tight closure of $N$ in $M$. 
\end{definition}

\begin{theorem}\cite[Theorem~10.2]{SharpBigTestElements}
\label{thm:Sharp-big-test-elements}
Let $S$ be an excellent regular ring of prime characteristic $p > 0$ that is $F$-intersection flat. If $R$ is a homomorphic image of $S$ regular in codimension $0$, then for any $c \in R$ not contained in any minimal prime such that $R_c$ is regular, some power of $c$ is a big test element of $R$.
\end{theorem}

\section{Excellence and \emph{F}-intersection flatness of dagger Tate algebras}

We fix a NA field $(k,|\cdot|)$ and an algebraic closure $\overline{k}$ of $k$. We also denote by $|\cdot|$ the unique NA norm on $\overline{k}$ that extends the norm on $k$. We let 
\[\Gamma \coloneqq |\overline{k}^\times|,\]
that is, $\Gamma$ is the value group of $\overline{k}$. Note that $\Gamma$ is always a dense subgroup of $\mathbb{R}_{> 0}$. For $r \in \mathbb{R}_{>0}$, we let $\rho_r \in (\mathbb{R}_{>0})^n$ denote the $n$-tuple all of whose coordinates are $r$. Recall that $T_{n,\rho_r}(k)$ is an excellent regular $k$-algebra (\autoref{prop:convergent-powerseries-general-polydisc}\autoref{prop:convergent-powerseries-general-polydisc.d}), and if $r \in \Gamma$, then all the maximal ideals of $T_{n,\rho_r}(k)$ have height $n$ (\autoref{prop:convergent-powerseries-general-polydisc}\autoref{prop:convergent-powerseries-general-polydisc.a}). Moreover, if $r > 1$, then $T_{n,\rho_r}(k)$ is a $k[X_1,\dots,X_n]$-subalgebra of $T_n(k)$ such that $T_{n,\rho_r}(k) \hookrightarrow T_n(k)$ is a regular map (\autoref{prop:convergent-powerseries-general-polydisc}\autoref{prop:convergent-powerseries-general-polydisc.e}). It is clear that if $1 < r < s$, then $T_{n,\rho_s}(k) \subset T_{n,\rho_r}(k)$. Using the fact that $\Gamma$ is dense in $\mathbb{R}_{> 0}$ we see that the collection 
\[\{T_{n,\rho_r}(k) \colon r \in \Gamma_{> 1}\}\]
is filtered by inclusion.

\begin{definition}
\label{def:Dagger-Tate}
For every positive integer $n$, the \emph{dagger Tate algebra} or the \emph{overconvergent Tate algebra} in $n$ indeterminates over $k$, denoted, $T_n(k)^\dagger$, is the filtered union/colimit
\[T_n(k)^\dagger \coloneqq \bigcup_{r \in \Gamma_{> 1}} T_{n,\rho_r}(k).\]
A \emph{dagger affinoid $k$-algebra} is a homomorphic image of some  $T_n(k)^\dagger$.
\end{definition}

Clearly $T_n(k)^\dagger$ is a $k[X_1,\dots,X_n]$-sub algebra of $T_n(k)$. It is a normed $k$-algebra under the restriction of the Gauss norm on $T_n(k)$, but not a Banach algebra, since it is not complete.

\begin{remark}
\label{rem:notation-in-GK}
We will often refer to \cite{GKDaggeralgebras}, where some basic properties of dagger Tate algebras are discussed. However, this paper uses the notation $W_n$ instead of $T_n(k)^\dagger$.
\end{remark}

\begin{remark}
\label{rem:properties-dagger-Tate}
We now summarize some properties of $T_n(k)^\dagger$ with indications of proofs.
\begin{enumerate}[itemsep = 1mm]
    \item\label{properties-dagger-Tate.completion} The completion of $T_n(k)^\dagger$ with respect to the Gauss norm is $T_n(k)$ \cite[Proposition 1.6]{GKDaggeralgebras}.
    \item\label{properties-dagger-Tate.JacobsonUFD} $T_n(k)^\dagger$ is a Jacobson UFD \cite[1.4]{GKDaggeralgebras}.
    \item\label{properties-dagger-Tate.NoetherNormalization} Every dagger affinoid $k$-algebra $A$ admits a finite injection $T_n(k)^\dagger \hookrightarrow A$ for some $n$ \cite[1.4]{GKDaggeralgebras}.
    \item\label{properties-dagger-Tate.faithfulflatness} $T_n(k)^{\dagger} \hookrightarrow T_n(k)$ is faithfully flat. For flatness one uses \cite[\href{https://stacks.math.columbia.edu/tag/05UU}{Tag 05UU, part (1)}]{stacks-project} and the fact that for all $r \in \Gamma_{> 1}$, $T_{n,\rho_r}(k) \hookrightarrow T_n(k)$ is flat by \autoref{prop:convergent-powerseries-general-polydisc}\autoref{prop:convergent-powerseries-general-polydisc.e}. To show faithful flatness it is enough to show $T_n(k)^{\dagger} \hookrightarrow T_n(k)$ is surjective on $\MSpec$, which follows by \cite[Proposition 1.5]{GKDaggeralgebras}.
    \item\label{properties-dagger-Tate.regular} $T_n(k)^\dagger$ is a Noetherian regular ring. Indeed, this follows by the regularity of $T_n(k)$ and faithfully flat descent \autoref{properties-dagger-Tate.faithfulflatness}.
    \item\label{properties-dagger-Tate.maximal.ideals} Let $\eta$ be a maximal ideal of $T_n(k)^\dagger$. Then $\eta T_n(k)$ is a maximal ideal of $T_n(k)$ and $\eta' \coloneqq \eta \cap k[X_1,\dots,X_n]$ is a maximal ideal of $k[X_1,\dots,X_n]$ such that $\eta' T_n(k)^\dagger = \eta$. Moreover, we have isomorphisms 
    \[k[X_1,\dots,X_n]/\eta' \xrightarrow{\cong} T_n(k)^\dagger/\eta \xrightarrow{\cong} T_n(k)/\eta T_n(k).\]
    To see this, choose a maximal ideal $\fram$ of $T_n(k)$ such that $\fram \cap T_n(k)^\dagger = \eta$. Then $\eta' = \fram \cap k[X_1,\dots,X_n]$ is maximal and $\eta'T_n(k) = \fram$ by \autoref{thm:Tate-algebras-properties}\autoref{thm:Tate-maximal-ideal-expanded}. Thus, $\eta T_n(k) = \fram$ as well, so $\eta T_n(k)$ is maximal. Moreover, by faithful flatness of $T_n(k)^\dagger \hookrightarrow T_n(k)$ one then has $\eta = \eta T_n(k) \cap T_n(k)^\dagger = \fram \cap T_n(k)^\dagger = (\eta' T_n(k)^\dagger)T_n(k) \cap T_n(k)^\dagger = \eta' T_n(k)^\dagger$. The assertion about isomorphism of residue fields now follows by the isomorphism $k[X_1,\dots,X_n]/\eta' \cong T_n(k)/\fram$ of \autoref{thm:Tate-algebras-properties}\autoref{thm:Tate-maximal-ideal-expanded}.
    \item\label{properties-dagger-Tate.maximal.dimension} All maximal ideals of $T_n(k)^\dagger$ are generated by $n$ elements and have height $n$. Indeed, the assertion about the number of generators follows by \autoref{properties-dagger-Tate.maximal.ideals} and the corresponding fact for number of generators of maximal ideals of polynomial rings. Let $\eta$ be a maximal ideal of $T_n(k)^\dagger$. Let $\fram = \eta T_n(k)$ be the unique maximal ideal of $T_n(k)$ lying over $\eta$ by \autoref{properties-dagger-Tate.maximal.ideals}. Then $T_n(k)^\dagger_\eta \hookrightarrow T_n(k)_\fram$ is a flat local homomorphism of Noetherian local rings whose closed fiber ring is a field. Then $\dim(T_n(k)^\dagger_\eta) = \dim(T_n(k)_\fram) = n$ by \cite[\href{https://stacks.math.columbia.edu/tag/00ON}{Tag00ON}]{stacks-project} and \autoref{thm:Tate-algebras-properties}\autoref{thm:Tate-dimension}.
\end{enumerate}
\end{remark}

We would next like to show that $T_n(k)^\dagger \hookrightarrow T_n(k)$ is regular and hence deduce that $T_n(k)^\dagger$ is excellent. We note that excellence of  dagger affinoid algebras is asserted in \cite[Theorem~1.7]{GKDaggeralgebras} without a detailed proof.

\begin{proposition}
\label{prop:Dagger-affinoid-regular-to-Tate-excellent}
Let $(k,|\cdot|)$ be a NA field. Then $T_n(k)^\dagger \hookrightarrow T_n(k)$ is faithfully flat and regular. Consequently, $T_n(k)^\dagger$, and hence any dagger affinoid $k$-algebra, is excellent.
\end{proposition}

\begin{proof}
We have $T_n(k)^\dagger = \colim_{r \in \Gamma_{> 1}} T_{n,\rho_r}(k)$ is a Noetherian regular ring (see \autoref{rem:properties-dagger-Tate}\autoref{properties-dagger-Tate.regular}). Moreover, $T_{n,\rho_r} \to T_n(k)$ is regular for all $r \in \Gamma_{> 1}$ by \autoref{prop:convergent-powerseries-general-polydisc}\autoref{prop:convergent-powerseries-general-polydisc.e}. Thus, $T_n(k)^\dagger \hookrightarrow T_n(k)$ is regular by \autoref{cor:colimit-regular-is-regular}.

We have $T_n(k)^\dagger \hookrightarrow T_n(k)$ is faithfully flat by \autoref{rem:properties-dagger-Tate}\autoref{properties-dagger-Tate.faithfulflatness}. Since $T_n(k)$ is quasi-excellent (\autoref{thm:Tate-algebras-properties}\autoref{thm:Tate-excellent}) and Cohen-Macaulay, we get $T_n(k)^\dagger$ is excellent by \autoref{lem:descent-quasi-excellence}. The excellence of dagger affinoid $k$-algebras now follows because homomorphic images of excellent rings are excellent.
\end{proof}

We are now ready to prove the main result about dagger Tate algebras.

\begin{theorem}
\label{thm:Dagger-Tate-FIF}
Let $(k,|\cdot|)$ be a NA field of characteristic $p > 0$. For all positive integers $n$ and $r, s \in \mathbb{N}$, if $\underline{x} = x_1,\dots,x_r$ and $\underline{y} = y_1,\dots,y_s$ are indeterminates, then for a multiplicative set $W \subset T_n(k)^\dagger[\underline{x}]$, $\left(W^{-1}T_n(k)^\dagger[\underline{x}]\right)\llbracket \underline{y}
\rrbracket$ is $F$-intersection flat.
\end{theorem}

\begin{proof}
Since $T_n(k)^\dagger \hookrightarrow T_n(k)$ is faithfully flat and regular (\autoref{prop:Dagger-affinoid-regular-to-Tate-excellent}), so is the essentially of finite type base change $\left(W^{-1}T_n(k)^\dagger[\underline{x}]\right)[\underline{y}] \to \left(W^{-1}T_n(k)[\underline{x}]\right)[\underline{y}]$ by \cite[\href{https://stacks.math.columbia.edu/tag/07C1}{Tag 07C1}, \href{https://stacks.math.columbia.edu/tag/07C0}{Tag 07C0}]{stacks-project}.
Since $\left(W^{-1}T_n(k)^\dagger[\underline{x}]\right)[\underline{y}]$ is excellent by \autoref{prop:Dagger-affinoid-regular-to-Tate-excellent}, the $(\underline{y})$-adic completion $\left(W^{-1}T_n(k)^\dagger[\underline{x}]\right)\llbracket \underline{y} \rrbracket \to \left(W^{-1}T_n(k)[\underline{x}]\right)\llbracket \underline{y} \rrbracket$ is faithfully flat (\autoref{rem:faithful-flatness-completion}) and regular  (\autoref{prop:completions-regular-maps}). The result now follows by descent (\autoref{prop:descent-FIF}) and the $F$-intersection flatness of $\left(W^{-1}T_n(k)[\underline{x}]\right)\llbracket \underline{y} \rrbracket$. This last fact is mentioned in \cite[Remarks~5.5.6 (a)]{DESTate}, and its proof is similar to the one given later for a more general assertion shown in  \autoref{thm:berktate:intflat}.
\end{proof}

\begin{corollary}
\label{cor:Test-elements-dagger-affinoid}
Let $(k,|\cdot|)$ be a NA field of characteristic $p > 0$. Let $R$ be essentially of finite type over $T_n(k)^\dagger$ and regular in codimension $0$. Let $I$ be any ideal of $R$ ($I = 0$ is allowed). Then $\widehat{R}^I$ has a big test element. In fact, if $c \in R$ is not contained in any minimal primes of $R$ and $R_c$ is regular, then some power of $c$ is a big test element of $\widehat{R}^I$.
\end{corollary}

\begin{proof}
The excellence of $R$ implies that $R \to \widehat{R}^I$ is a regular ring map \cite[\href{https://stacks.math.columbia.edu/tag/0AH2}{Tag 0AH2}]{stacks-project}, and hence $\widehat{R}^I$ is regular in codimension $0$ by \cite[\href{https://stacks.math.columbia.edu/tag/033A}{Tag 033A}]{stacks-project}. If $c \in R$ is not contained in any minimal prime of $R$, then its image in $\widehat{R}^I$ cannot be contained in any minimal prime of $\widehat{R}^I$ since minimal primes $\widehat{R}^I$ contract to minimal primes of $R$ by the going-down property of flat maps. Furthermore, if $R_c$ is a regular ring, then so is $(\widehat{R}^I)_c$ \cite[\href{https://stacks.math.columbia.edu/tag/0H7S}{Tag 0H7S}]{stacks-project}. Then the result follows by \autoref{thm:Sharp-big-test-elements} because $\widehat{R}^I$ is a homomorphic image of an $F$-intersection flat (\autoref{thm:Dagger-Tate-FIF}) excellent regular ring of the form $\left(W^{-1}T_n(k)^\dagger[x_1,\dots,x_r]\right)\llbracket y_1,\dots,y_s \rrbracket$ for a multiplicative set $W \subset T_n(k)^\dagger[x_1,\dots,x_r]$.
%, and such a localization is $F$-intersection flat because $T_n(k)^\dagger[x_1,\dots,x_k]\llbracket y_1,\dots,y_\ell \rrbracket$ is $F$-intersection flat by \autoref{thm:Dagger-Tate-FIF} and because $F$-intersection flatness is preserved under localization \cite[Theorem~3.1.3]{DESTate}.
\end{proof}

\section{\emph{F}-intersection flatness of the ring of convergent power series in the polydisc \(\mathbb{B}^n_\rho(k)\)}

\begin{defn}
    Let \((k,|\cdot|)\) be a non-Archimedean field, and let \((E,||\cdot||_E)\) be a normed space over \(k\). Let \(\rho \in (\mathbb{R}_{> 0})^n\). Throughout the following, $\nu$ denotes a multi-index $(\nu_1, \ldots, \nu_n) \in \mathbb{N}^n$, and $\rho^\nu$ denotes 
    %$(\rho_1^{i_1}, \ldots, \rho_n^{i_n})$.
    $\rho_1^{\nu_1} \cdots \rho_n^{\nu_n}$.
    A (multi-)sequence \((a_{\nu})_{\nu \in \mathbb{N}^n}\) in \(E\) is \(\rho\)-null if
    \(||a_{\nu}||_E\rho^\nu \xrightarrow{} 0\) as $\nu_1 + \dots + \nu_n \to \infty$.
\end{defn}

\begin{lem}
    Let \((k,|\cdot|)\) be a non-Archimedean field, and let \((E,||\cdot||_E)\) and \((F,||\cdot||_F)\) be normed spaces over \(k\).
    Suppose \(f \colon E \xrightarrow{} F\) is a \(k\)-linear continuous map and $\rho \in (\mathbb{R}_{> 0})^n$.
    Then \(f\) maps \(\rho\)-null sequences to \(\rho\)-null sequences.
\end{lem}

\begin{proof}
    Since \(f\) is continuous, it is bounded (see for instance \cite[Lem.\ 2.11]{DattaMurayamaTate}). Thus, there exists $M > 0$ such that for all $a \in E$, $||f(a)||_F \leq M||a||_E$.
    Then for any $\rho$-null sequence $(a_\nu)_{\nu \in \mathbb{N}^n}$ of $E$, we have  \(||f(a_{\nu})||_F\rho^{\nu} \leq M||a_{\nu}||_E\rho^{\nu} \xrightarrow{} 0\) as $\nu_1+\dots +\nu_n \to \infty$.
\end{proof}

\begin{corollary}
    \label{cor:cont:berktate}
    Let $(k,|\cdot|_k) \hookrightarrow (\ell,|\cdot|_\ell)$ be an extension of non-Archimedean fields. Let $\rho \in (\mathbb{R}_{> 0})^n$. Then any continuous $k$-linear map \(\ell \xrightarrow{} k\)
    induces a $T_{n,\rho}(k)$-linear map 
    \(T_{n,\rho}(\ell) \xrightarrow{} T_{n,\rho}(k)\).
\end{corollary}

\begin{proof}
    By definition of the \(\rho\)-Gauss norm, such a map need only take
    \(\rho\)-null sequences to  \(\rho\)-null sequences
\end{proof}

\begin{set}
    \label{set:k:ell}
    let $\rho \in (\mathbb{R}_{> 0})^n$, and let \((k,|\cdot|_k) \inj (\ell,|\cdot|_\ell)\) be an extension of non-Archimedean fields,
    such that either
    \begin{itemize}
        \item \(k\) is spherically complete, or
        \item \(\ell\) has a dense subspace over \(k\) 
            with a countable basis.
    \end{itemize}
    Recall a valued-field $(k,|\cdot|)$ is \emph{spherically complete} if for every decreasing sequence of closed non-empty balls
    \[
    D_1 \supseteq D_2 \supseteq D_3 \supseteq \dots,    
    \]
    of $k$, the intersection $\bigcap_{i=1}^\infty D_i$ is also non-empty. The point here is that we do not require the balls $D_n$ to have a common center.
\end{set}

\begin{lem}
    \label{lem:berktate:split}
    Let \(k, \ell\) as in \autoref{set:k:ell}.
    Then \(T_{n,\rho}(k) \xrightarrow{} T_{n,\rho}(\ell)\) 
    is a split extension of \(T_{n,\rho}(k)\)-algebras.
\end{lem}

\begin{proof}
    Following \cite[Theorem~5.4.1]{DESTate},
    we have, by \cite[Theorem~5.2.5~(a)]{DESTate} or
    \cite[Theorem~5.2.9]{DESTate}, a continuous
    linear functional \(\ell \xrightarrow{} k\) 
    extending the identity on \(k\).
    By \autoref{cor:cont:berktate},
    this gives a bounded map
    \(T_{n,\rho}(\ell) \xrightarrow{} T_{n,\rho}(k)\)
    which is visibly a splitting
    (as it maps 1 to 1).
\end{proof}

\begin{proposition}
    \label{prop:berktate:ort}
    Let \(k, \ell\) be as in \autoref{set:k:ell}.
    Then \(T_{n,\rho}(k) \xrightarrow{} T_{n,\rho}(\ell)\)
    is Ohm-Rush trace.
\end{proposition}

\begin{proof}[Sketch of proof]
    We need to show that \(f \in T_{n,\rho}(\ell)\) 
    is in the ideal \(\Tr(f)T_{n,\rho}(\ell)\).
    We will show, following, \cite[Theorem~5.4.1]{DESTate},
    that there exists a real number \(C\), with \(0 < C\), such that for every  \(\epsilon\) with \(0 < \epsilon\),
    there is a \(g_{\epsilon} \in \Tr(f)T_{n,\rho}(\ell)\) 
    such that \(||f - g_{\epsilon}|| < C \epsilon\). This implies the theorem because ideals of $T_{n,\rho}$ are closed under the Gauss topology.

    Now,
    for \(f = \sum_{\nu \in \mathbb{N}^{n}}^{} a_{\nu}X^{\nu} \in T_{n,\rho}(\ell)\),
    let 
    \[f_{\epsilon} = \sum_{\epsilon \leq |a_{\nu}|_\ell\rho^{\nu}}^{} a_{\nu}X^{\nu},\]
    so that \(||f - f_{\epsilon}|| < \epsilon\).
    Then the argument from \cite[Theorem~5.4.1]{DESTate},
    paragraph 4, until Equation 5.4.3.5,
    applies verbatim, so we have
    \begin{itemize}
        \item a $k$-vector subspace \(\ell_{\epsilon}\) of $\ell$
            spanned by the coefficients of \(f_{\epsilon}\).
        \item a basis \(\{x_{\epsilon,1},\dots, x_{\epsilon,n_{\epsilon}}\}\) 
            of \(\ell_{\epsilon}\) 
            such that \(1 \leq  |x_{\epsilon,i}|_\ell \leq M\), 
            where \(M\) does not depend on \(i\) or \(\epsilon\).
        \item continuous linear functionals
            \(\widetilde{x_{\epsilon,i}^{\star}} \colon \ell \to k\) extending
            the dual basis vectors \(x_{\epsilon,i}^{\star} \colon \ell_\epsilon \to k\) 
            for \(i = 1, \ldots, n_{\epsilon}\),
            such that 
            \(|| \widetilde{x_{\epsilon,i}^{\star}} || \leq 4.\) Here we are using the \emph{operator norm} or \emph{Lipschitz norm} on the linear functionals.
    \end{itemize}

    Still following \cite[Theorem~5.4.1]{DESTate},
    by \autoref{cor:cont:berktate} the functionals
    \(\widetilde{x_{\epsilon,i}^{\star}}\) induce \(T_{n,\rho}(k)\)-linear
    maps \(\varphi_{\epsilon,i} \colon T_{n,\rho}(\ell) \to T_{n,\rho}(k)\) by applying the functional
    coefficient-wise.
    We define \(g_{\epsilon}\) to be \(\sum_{i=1}^{n_{\epsilon}} \varphi_{\epsilon,i}(f) x_{\epsilon,i}
    \in \Tr(f)T_{n,\rho}(\ell)\).
    Now, the rest of the proof of \cite[Theorem~5.4.1]{DESTate},
    applies, verbatim, with the quantity \(|a_{\nu}|_{\ell}\) 
    replaced by \(|a_{\nu}|_{\ell}\rho^{\nu}\) to show that 
    \(||f - g_{\epsilon}|| < 4M\epsilon\).
\end{proof}

\begin{corollary}
\label{cor:preserving-ORT-adjoining-variables}
Let \(k, \ell\) be as in \autoref{set:k:ell}, $\underline{x} = x_1,\dots,x_r$ and $\underline{y} = y_1,\dots,y_s$ be indeterminates, and $W \subset T_{n,\rho}(k)[\underline{x}]$ be a multiplicative set. Then $\left(W^{-1}T_{n,\rho}(k)[\underline{x}]\right)\llbracket\underline{y}\rrbracket \to \left(W^{-1}T_{n,\rho}(\ell)[\underline{x}]\right)\llbracket\underline{y}\rrbracket$ is Ohm-Rush trace.
\end{corollary}

\begin{proof}
By \autoref{prop:berktate:ort} and \cite[Proposition~4.1.9 (1)]{DattaEpsteinTuckerML}, we get $W^{-1}T_{n,\rho}(k)[\underline{x}] \to W^{-1}T_{n,\rho}(\ell)[\underline{x}]$ is Ohm-Rush trace. Then $\left(W^{-1}T_{n,\rho}(k)[\underline{x}]\right)\llbracket\underline{y}\rrbracket \to \left(W^{-1}T_{n,\rho}(\ell)[\underline{x}]\right)\llbracket\underline{y}\rrbracket$ is Ohm-Rush trace by \cite[Proposition~4.1.9~(3)]{DattaEpsteinTuckerML}.
\end{proof}

One can also give a different proof of faithful flatness of extensions of Berkovich Tate algebras for extensions of NA fields in \autoref{set:k:ell}.

\begin{corollary}
    \label{cor:berktate:fflat}
    Let \(\ell, k\) as in \autoref{set:k:ell}.
    Then \(T_{n,\rho}(k) \xrightarrow{} T_{n,\rho}(\ell)\) 
    is faithfully flat.
\end{corollary}

\begin{proof}
The map is flat by \autoref{prop:berktate:ort} because Ohm-Rush trace modules are flat. As the map also splits by  \autoref{lem:berktate:split}, it is faithfully flat.
\end{proof}

\begin{theorem}
    \label{thm:berktate:intflat}
    Let \((k,|\cdot|)\) be a non-Archimedean field, $n \in \ZZ_{>0}$, and $\rho \in (\mathbb{R}_{>0})^n$. Fix an algebraic closure $\overline{k}$ of $k$ and let $\widehat{\overline{k}}$ be the completion of $\overline{k}$ with respect to the unique norm that extends the norm on $k$. Let $\underline{x} = x_1,\dots,x_r$ and $\underline{y} = y_1,\dots,y_s$ be indeterminates. Let $W \subset T_{n,\rho}(k)[\underline{x}]$ be a multiplicative set. Then \(\left(W^{-1}T_{n,\rho}(\widehat{\overline{k}})[\underline{x}]\right)\llbracket\underline{y}\rrbracket\) is an intersection flat \(\left(W^{-1}T_{n,\rho}(k)[\underline{x}]\right)\llbracket\underline{y}\rrbracket\)-algebra.
\end{theorem}

\begin{proof}
    Following \cite[Theorem~5.5.1]{DESTate}, consider the collection \(\Sigma\) of algebraic extensions of \(k\) contained in $\overline{k}$ which have a countable basis, and the collection \(\Sigma^{\prime}\) of completions of elements of \(\Sigma\), which are subfields of \(\widehat{\overline{k}}\). It is straightforward to verify that $\Sigma'$ is filtered under inclusion. Furthermore, by an argument similar to the one given in the proof of \cite[Theorem~5.5.1]{DESTate}, we get 
    \(\colim_{\ell \in \Sigma^{\prime}} T_{n,\rho}(\ell) 
    = \bigcup_{\ell \in \Sigma^{\prime}} T_{n,\rho}(\ell) = T_{n,\rho}(\widehat{\overline{k}}),\)
    and consequently, also that 
    \[\colim_{\ell \in \Sigma^{\prime}} \left[ \left(W^{-1}T_{n,\rho}(\ell)[\underline{x}]\right)\llbracket\underline{y}\rrbracket \right]
    = \bigcup_{\ell \in \Sigma^{\prime}} \left[\left(W^{-1}T_{n,\rho}(\ell)[\underline{x}]\right)\llbracket\underline{y}\rrbracket \right] = \left(W^{-1}T_{n,\rho}(\widehat{\overline{k}})[\underline{x}]\right)\llbracket\underline{y}\rrbracket .\]
    Indeed, for the last equality note that any element of $\left(W^{-1}T_{n,\rho}(\widehat{\overline{k}})[\underline{x}]\right)\llbracket\underline{y}\rrbracket$ uses only countably many coefficients in $\widehat{\overline{k}}$, which consequently must lie in the completion of an algebraic sub-extension of $k$ in $\widehat{\overline{k}}$ with a countable basis. 
    Now, for \(\ell, \ell^{\prime} \in \Sigma'\) with \(\ell \ins \ell^{\prime}\),
    by \autoref{cor:berktate:fflat} and \autoref{rem:faithful-flatness-completion} we have 
    \(\left(W^{-1}T_{n,\rho}(\ell)[\underline{x}]\right)\llbracket\underline{y}\rrbracket \inj \left(W^{-1}T_{n,\rho}(\ell')[\underline{x}]\right)\llbracket\underline{y}\rrbracket\) is faithfully flat because $\ell'$ also has a dense subspace over $\ell$ with a countable basis. By \autoref{cor:preserving-ORT-adjoining-variables} each $\left(W^{-1}T_{n,\rho}(\ell)[\underline{x}]\right)\llbracket\underline{y}\rrbracket$ is an Ohm-Rush trace, and hence, an intersection flat $\left(W^{-1}T_{n,\rho}(k)[\underline{x}]\right)\llbracket\underline{y}\rrbracket$-algebra (see \cite[Proposition~4.3.8]{DattaEpsteinTuckerML} for the fact that Ohm-Rush trace implies intersection flat).
    Then
    \(\colim_{\ell \in \Sigma^{\prime}} \left[\left(W^{-1}T_{n,\rho}(\ell)[\underline{x}]\right)\llbracket\underline{y}\rrbracket \right] = \left(W^{-1}T_{n,\rho}(\widehat{\overline{k}})[\underline{x}]\right)\llbracket\underline{y}\rrbracket\)
    is also an intersection flat as a \(\left(W^{-1}T_{n,\rho}(k)[\underline{x}]\right)\llbracket\underline{y}\rrbracket\)-algebra because intersection flatness is preserved under filtered colimits with universally injective transition maps by \cite[Corollary~4.3.2]{DattaEpsteinTuckerML}.
\end{proof}

\begin{corollary}
    \label{cor:berktate:FIF} 
    Let \((k,|\cdot|)\) be a non-Archimedean field, $n \in \ZZ_{>0}$, and $\rho \in (\mathbb{R}_{>0})^n$.
   If $k$ has characteristic $p> 0$, then for indeterminates $\underline{x}= x_1,\dots,x_r$, $\underline{y} = y_1,\dots,y_s$, and a multiplicative set $W \subset T_{n,\rho}(k)[\underline{x}]$, we have that \(\left(W^{-1}T_{n,\rho}(k)[\underline{x}]\right)\llbracket\underline{y}\rrbracket\) is \(F\)-intersection flat. Consequently, if $R$ is an essentially finite type $T_{n,\rho}(k)$-algebra that is regular in codimension $0$ and $I \subset R$ is any ideal, then for any $c \in R$ not contained in any minimal prime such that $R_c$ is regular, some power of the image of $c$ in $\widehat{R}^I$ is a big test element of $\widehat{R}^I$.
\end{corollary}

\begin{proof}
    $\widehat{R}^I$ is regular in codimension $0$ and $(\widehat{R}^I)_c$ is regular by a reasoning similar to the one given in the proof of \autoref{cor:Test-elements-dagger-affinoid}. Since $\widehat{R}^I$ is a homomorphic image of an excellent $F$-intersection flat ring of the form \(\left(W^{-1}T_{n,\rho}(k)[\underline{x}]\right)\llbracket\underline{y}\rrbracket\), some power of the image of $c$ in $\widehat{R}^I$ is a big test element by \autoref{thm:Sharp-big-test-elements}. For the $F$-intersection flatness assertion, consider the flat extensions 
    \begin{align*}
    \left(W^{-1}T_{n,\rho}(k)[\underline{x}]\right)\llbracket\underline{y}\rrbracket &\hookrightarrow \left(W^{-1}T_{n,\rho}(k^{1/p})[\underline{x}]\right)\llbracket\underline{y}\rrbracket \hookrightarrow \left(W^{-1}(T_{n,\rho}(k))^{1/p}[\underline{x}^{1/p}]\right)\llbracket\underline{y}\rrbracket\\
    &\hookrightarrow \left(W^{-1}(T_{n,\rho}(k))^{1/p}[\underline{x}^{1/p}]\right)\llbracket\underline{y}^{1/p}\rrbracket = \left(\left(W^{-1}T_{n,\rho}(k)[\underline{x}]\right)\llbracket\underline{y}\rrbracket\right)^{1/p}.
    \end{align*}
    The final extension is free of rank $p^s$ and hence Ohm-Rush trace \cite[Lemma~4.1.5]{DattaEpsteinTuckerML}. Since $T_{n,\rho}(k^{1/p}) \hookrightarrow (T_{n,\rho}(k))^{1/p}$ is free (of rank $p^n$), $W^{-1}T_{n,\rho}(k^{1/p})[\underline{x}] \hookrightarrow W^{-1}(T_{n,\rho}(k))^{1/p}[\underline{x}^{1/p}]$ is free (of rank $p^{n+r}$). Hence $\left(W^{-1}T_{n,\rho}(k^{1/p})[\underline{x}]\right)\llbracket\underline{y}\rrbracket \hookrightarrow \left(W^{-1}(T_{n,\rho}(k))^{1/p}[\underline{x}^{1/p}]\right)\llbracket\underline{y}\rrbracket$ is Ohm-Rush trace by \cite[Proposition~4.1.9~(3)]{DattaEpsteinTuckerML}. Thus, the second and third maps in the above composition are intersection flat by \cite[Proposition~4.3.8]{DattaEpsteinTuckerML}. Since composition of intersection flat ring maps is intersection flat, it suffices to show that the first map is intersection flat. Choosing an embedding $k^{1/p} \hookrightarrow \widehat{\overline{k}}$, we get $\left(W^{-1}T_{n,\rho}(k^{1/p})[\underline{x}]\right)\llbracket\underline{y}\rrbracket \hookrightarrow \left(W^{-1}T_{n,\rho}(\widehat{\overline{k}})[\underline{x}]\right)\llbracket\underline{y}\rrbracket$ is a faithfully flat ring map, and hence, by restriction of scalars, a universally injective $\left(W^{-1}T_{n,\rho}(k)[\underline{x}]\right)\llbracket\underline{y}\rrbracket$-linear map. Since $\left(W^{-1}T_{n,\rho}(\widehat{\overline{k}})[\underline{x}]\right)\llbracket\underline{y}\rrbracket$ is an intersection flat $\left(W^{-1}T_{n,\rho}(k)[\underline{x}]\right)\llbracket\underline{y}\rrbracket$-algebra (\autoref{thm:berktate:intflat}), by universally injective descent of intersection flatness \cite[Corollary~4.3.2 (2)]{DattaEpsteinTuckerML}, $\left(W^{-1}T_{n,\rho}(k^{1/p})[\underline{x}]\right)\llbracket\underline{y}\rrbracket$ is an intersection flat $\left(W^{-1}T_{n,\rho}(k)[\underline{x}]\right)\llbracket\underline{y}\rrbracket$-algebra.  
   %where $\underline{x}^{1/p} = x_1^{1/p},\dots,x_r^{1/p}$, and similarly for $\underline{y}^{1/p}$. The extension $T_{n,\rho}(k^{1/p})[\underline{x}]\llbracket\underline{y}\rrbracket \hookrightarrow (T_{n,\rho}[\underline{x}]\llbracket\underline{y}\rrbracket)^{1/p}$ is free, hence Ohm-Rush trace (\cite[Lemma~4.1.5 (2)]{DattaEpsteinTuckerML}), and hence intersection flat. Since the property of being intersection is preserved under composition, it suffices to $T_{n,\rho}(k)[\underline{x}]\llbracket\underline{y}\rrbracket \hookrightarrow T_{n,\rho}(k^{1/p})[\underline{x}]\llbracket\underline{y}\rrbracket$ is intersection flat. Choosing an embedding $k^{1/p} \hookrightarrow \widehat{\overline{k}}$, we get $T_{n,\rho}(k^{1/p})[\underline{x}]\llbracket\underline{y}\rrbracket \hookrightarrow T_{n,\rho}(\widehat{\overline{k}})[\underline{x}]\llbracket\underline{y}\rrbracket$ is a faithfully flat ring map, and hence, by restriction of scalars, a universally injective $T_{n,\rho}(k)[\underline{x}]\llbracket\underline{y}\rrbracket$-linear map. Since $T_{n,\rho}(\widehat{\overline{k}})[\underline{x}]\llbracket\underline{y}\rrbracket$ is an intersection flat $T_{n,\rho}(k)[\underline{x}]\llbracket\underline{y}\rrbracket$-algebra (\autoref{thm:berktate:intflat}), by universally injective/pure descent of intersection flatness (\cite[Corollary~4.3.2 (2)]{DattaEpsteinTuckerML}), we get $T_{n,\rho}(k^{1/p})[\underline{x}]\llbracket\underline{y}\rrbracket$ is an intersection flat $T_{n,\rho}(k)[\underline{x}]\llbracket\underline{y}\rrbracket$-algebra, as desired.  
\end{proof}

% From Proposition \ref{thm:berktate:intflat}, we deduce that
   %  for any algebraic extension \(\ell\) of \(k\), 
   %  \(T_{n,\rho}(k) \xrightarrow{} T_{n,\rho}(\ell)\) is
   %  intersection flat.
   %  Then, we consider the Frobenius on \(T_{n,\rho}(k)\),
   %  which factors as
   %  \(T_{n,\rho}(k) \inj T_{n,\rho}(k^{1 / p}) \inj (T_{n,\rho}(k))^{1 / p}\).
   %  The second extension is free, and the first is intersection flat
   %  as \(k^{1 / p}\) is algebra.
   %  We conclude by \cite[Remark~4.2.3~(d)]{DattaEpsteinTuckerML}. 

%% PUT BACK AT THE END
\section{Acknowledgements}
Many ideas of this paper originated in the related works \cite{DESTate,DattaEpsteinTuckerML} of the first and third authors with Neil Epstein and Karl Schwede. The authors thank them for helpful conversations. The first author also thanks Takumi Murayama for a helpful conversation about Grothendieck localization. The second author thanks Kiran Kedlaya and Nathan Wenger for helpful conversations. The second author would like to thank Tomer Schlank for 
%having the idea that inspired the idea that inspired the idea that inspired this paper.
inspiration.

\bibliographystyle{skalpha}
\bibliography{main,preprints}

\end{document}